\documentclass[11pt,reqno]{amsart}

\usepackage[utf8]{inputenc}
\usepackage[T1]{fontenc}
\usepackage{amsmath,amssymb,amsthm,mathtools}
\usepackage{microtype}
\usepackage{enumitem}
\usepackage{booktabs}
\usepackage{array}
\usepackage{hyperref}
\hypersetup{colorlinks=true,linkcolor=black,citecolor=black,urlcolor=black}

\theoremstyle{plain}
\newtheorem{theorem}{Theorem}[section]
\newtheorem{proposition}[theorem]{Proposition}
\newtheorem{lemma}[theorem]{Lemma}
\newtheorem{corollary}[theorem]{Corollary}

\theoremstyle{definition}
\newtheorem{definition}[theorem]{Definition}
\newtheorem{example}[theorem]{Example}

\theoremstyle{remark}
\newtheorem{remark}[theorem]{Remark}

\newcommand{\R}{\mathbb{R}}
\newcommand{\C}{\mathbb{C}}
\newcommand{\HH}{\mathbb{H}}
\newcommand{\OO}{\mathbb{O}}
\newcommand{\Z}{\mathbb{Z}}
\newcommand{\Q}{\mathbb{Q}}
\newcommand{\golden}{\varphi}
\newcommand{\IcosianRing}{\mathbb{I}}
\newcommand{\Order}{\mathcal{O}}
\newcommand{\Shell}{\mathcal{S}}
\newcommand{\Axis}{\mathcal{A}}
\newcommand{\Balanced}{\mathcal{D}}
\newcommand{\NN}{N}
\newcommand{\BB}{B}

\newcommand{\conv}{\mathrm{conv}}
\newcommand{\Aut}{\mathrm{Aut}}
\newcommand{\Mouf}{M}
\newcommand{\HI}{\mathrm{HI}}
\newcommand{\Hopfmap}{h}

\title{Integral Planes and Unit-Norm Polytopes}

\author{Daniele Corradetti}
\address{Grupo de F\'isica Matem\'atica\\
Instituto Superior T\'ecnico\\
Av.\ Rovisco Pais, 1049-001 Lisboa, Portugal}
\email{danielecorradetti@tecnico.ulisboa.pt}
\date{\today}

\begin{document}

\maketitle

\begin{abstract}
We introduce and study integral planes associated with crystallographic and non-crystallographic integral systems in real composition algebras. For an integral order $\Order$ in such an algebra we define the plane $\Order^{2}$ with quadratic form $Q(x,y)=\NN(x)+\NN(y)$, the axis shell, the balanced shell, and the corresponding unit-normalised spherical polytopes. For ten crystallographic orders we recover, in one uniform construction, the orthogonal-direct-sum root systems $2A_{1}$, $A_{2}\oplus A_{2}$, $4A_{1}$, $D_{4}\oplus D_{4}$, $16A_{1}$, and $E_{8}\oplus E_{8}$ (with classical-polytope realisations including the square, the 16-cell, the 24-cell, and the Gosset polytope $4_{21}$); for two non-crystallographic orders we obtain $H_{2}\oplus H_{2}$ (decagonal tegum) and $H_{4}\oplus H_{4}$ (600-cell tegum) over $\Z[\golden]$. We prove a rank-obstruction theorem that closes, unconditionally and by a purely Coxeter-theoretic argument, the existence question for an indecomposable rank-eight golden octonion order: no such order can exist. On the balanced shell side, we identify the genuine algebraic Hopf map $\Hopfmap_{A}(a,b)=(2a\bar b,\NN(a)-\NN(b))$ and prove that its restriction to the balanced shell is a finite principal fibration of the unit loop, valid both for the associative case and for the alternative Moufang case. 
\end{abstract}



\setcounter{tocdepth}{1}
\tableofcontents

\section{Introduction and Motivation}

Unital composition algebras, also known as Hurwitz algebras, are indispensable tools in geometry. Even more so the four real division Hurwitz algebras that, up to isomorphism, are the reals $\R$, the complexes $\C$, the quaternions $\HH$ and the octonions $\OO$. Inside each of them lives, in a canonical sense, an arithmetic skeleton: the Gaussian integers $\Z[i]$, the Eisenstein integers $\Z[\omega]$, the Hurwitz quaternions, and the Coxeter--Dickson integral octonions. These four arithmetic skeletons are familiar to algebraists, lattice theorists and crystallographers alike; they sit at the intersection of normed algebra, integral quadratic-form theory, and root-system combinatorics.

A more recent development, in part motivated by aperiodic order and the geometry of icosahedral symmetry (see Patera~\cite{Patera} for a survey), enlarges the picture. If one is willing to let the coefficient ring grow from $\Z$ to $\Z[\golden]$, where $\golden=(1+\sqrt{5})/2$, then the non-crystallographic Coxeter systems $H_{2}$, $H_{3}$, $H_{4}$ enter the same arithmetic stage; the relevant orders are the cyclotomic ring $\Z[\zeta_{10}]$ inside $\C$ and the icosian ring $\IcosianRing$ inside $\HH$. Building on the recent axiomatic framework of non-crystallographic integral root-shell systems over $\Z[\golden]$ established in~\cite{Corradetti2605}, we ask in this paper a complementary question: what spherical polytopes appear if, instead of treating $\Order$ as an algebra in its own right, one considers the integral \emph{plane} $\Order^{2}$ with the natural quadratic form $Q(x,y)=\NN(x)+\NN(y)$, and reads off the root system from the resulting unit shell?

The present work proposes a single construction that organises these crystallographic and non-crystallographic arithmetic skeletons into one shared geometric object. Given an integral order $\Order$ in a real composition algebra, with quadratic norm $\NN$ and polar form $\BB$, we consider the \emph{integral plane}
\[
\Order^{2}=\{(x,y):x,y\in\Order\},
\qquad
Q(x,y)=\NN(x)+\NN(y).
\]
This pair $(\Order^{2},Q)$ is an integral quadratic $R$-module whose unit-norm shell $\Shell_{1}(\Order^{2})$ is a finite spherical polytope. The same definition simultaneously yields the classical root-polytope doublings, $2A_{1}$ for $\Z$, $A_{2}\oplus A_{2}$ for the Eisenstein integers, $D_{4}\oplus D_{4}$ for the Hurwitz quaternions, $E_{8}\oplus E_{8}$ for the Coxeter--Dickson integral octonions, and their non-crystallographic counterparts $H_{2}\oplus H_{2}$ and $H_{4}\oplus H_{4}$.

In order to give an introduction we have to say that three themes run through the paper.

The first is the \emph{doubling theme}. The integral plane is a uniform construction that produces, in one stroke, the orthogonal-direct-sum root systems $R\oplus R$ for the root systems $R$ realisable as unit shells of integral orders. We separate, throughout, the algebraic structure (root system over $R$) from the geometric one (classical polytope, in the sense of Coxeter and Grünbaum). The unit shell of the Gaussian order, for example, is the root system $2A_{1}$ (two orthogonal copies of $A_{1}$, eight elements through length-doubling and centrality); as a polytope, its convex hull is the square, also known as the 2-cross polytope $\beta_{2}$. This distinction is sharpened in our synoptic Table~\ref{tab:main}, which carries one column for the root system and a separate column for the classical polytope.

The second is the \emph{obstruction theme}. Once the integral plane is in hand, a natural question is whether a strong indecomposable rank-eight golden octonion order can play the role of $\OO_{CD}$ in the non-crystallographic case. Earlier bounded computational searches reported in~\cite{Corradetti2605} were negative within their range. Here we close the question \emph{unconditionally}, by a purely Coxeter-theoretic argument: an indecomposable non-crystallographic root-shell system over $\Z[\golden]$ in the sense of Definition~\ref{def:root-shell-system} has rank at most four; in particular, no rank-eight indecomposable golden octonion order exists. We dedicate \S\ref{sec:rank-obstruction} to this rank-obstruction theorem, which is the main \emph{negative} result of the article.

The third is the \emph{fibration theme}. The balanced shell $\Balanced(\Order)=\tfrac{1}{\sqrt 2}\,\Shell_{1}(\Order)\times\Shell_{1}(\Order)$ carries the discrete restriction of the algebraic Hopf map $\Hopfmap_{A}(a,b)=(2a\bar b,\NN(a)-\NN(b))$ from $S^{2d-1}\subset A^{2}$ to $S^{d}\subset A\oplus\R$. Restricted to $\Balanced(\Order)$ the second coordinate vanishes and the image lies on the equator $S^{d-1}\times\{0\}\subset S^{d}$, recovering on this equator the principal fibration $\pi(u,v)=u\bar v$ of the unit loop $U=\Shell_{1}(\Order)$ over itself. We prove that this principal fibration is well-defined both in the associative and in the alternative Moufang case, the latter using Artin's two-generated associativity (Theorem~\ref{thm:hopf-moufang}). We deliberately distinguish this finite principal fibration from the  topological Hopf fibration $S^{2d-1}\to S^{d}$: the two have different codomains and different geometric content, and conflating them would be a mistake.

 The plane construction $\Order^{2}$ with $Q(x,y)=\NN(x)+\NN(y)$ sits at the crossroads of four otherwise separate strands: the classical lattices and orders (Conway and Sloane~\cite{ConwaySloane}; Conway and Smith~\cite{ConwaySmith}; Wilson~\cite{Wilson}); the non-crystallographic side (Moody and Patera~\cite{MoodyPatera}; Champagne, Kjiri, Moody and Patera~\cite{ChampagneKjiriMoodyPatera}); the literature on composition algebras and their projective planes (Springer--Veldkamp~\cite{SpringerVeldkamp}; Baez~\cite{Baez}; McCrimmon~\cite{McCrimmon}); and the recent axiomatic framework of non-crystallographic integral root-shell systems over $\Z[\golden]$ developed in~\cite{Corradetti2605}. The closest neighbour on the discrete-Hopf side is Castro Perelman~\cite{CastroPerelman}, who fibres the $E_{8}$ lattice over a $D_{4}$-cross polytope with $D_{4}$ fibres. What we contribute is a uniform plane construction that recovers the orthogonal-direct-sum root systems on both sides of the crystallographic--golden boundary; an unconditional rank-obstruction theorem closing the rank-eight golden octonion question; and the precise relationship between the algebraic Hopf map and its restriction to the balanced shell, valid in the alternative Moufang case via Artin's theorem.

Before any technical work, one might want to have a look at Table~\ref{tab:main} which summarises the classification proved in \S\S\ref{sec:crystal}--\ref{sec:hopf}. We introduced it here since one will benefit keeping it in sight: every later section corresponds to a row or a vertical slice of the table. We refer to it whenever a numerical claim is made in the text. The new feature, with respect to the standard catalogues of root-polytopes, is the separation of the column \emph{Root system} (an algebraic datum) from the column \emph{Classical polytope} (a geometric datum); the column \emph{Doubling root system} records the root system of the axis shell, and the column \emph{Doubling polytope} the classical name of its convex hull.

\begin{table}[t]
\centering
\caption{\textit{Synoptic table of the integral plane construction over the ten crystallographic and the two non-crystallographic orders studied in this article. The columns are: composition algebra; integral order $\Order$; cardinality of the unit shell $\lvert\Shell_{1}(\Order)\rvert$; root system of $\Shell_{1}(\Order)$; classical polytope realising the convex hull of $\Shell_{1}(\Order)$; root system of the axis shell $\Axis(\Order)$ ($=R\oplus R$); classical polytope of the axis-shell convex hull; cardinality of the balanced shell $\lvert\Balanced(\Order)\rvert$; algebraic type of the principal fibration. Here $\beta_{n}$ is the $n$-cross polytope (orthoplex); $\{m\}$ is the regular $m$-gon; $4_{21}$ is the Gosset polytope of $E_{8}$.}}
\label{tab:main}
\renewcommand{\arraystretch}{1.1}
\setlength{\tabcolsep}{3pt}
\footnotesize
\begin{tabular}{llccccccl}
\toprule
$A$ & Order $\Order$ & $\lvert\Shell_{1}\rvert$ & Root sys. & Polytope & Axis $R{\oplus}R$ & Axis polytope & $\lvert\Balanced\rvert$ & Fibration \\
\midrule
$\R$ & $\Z$ & $2$ & $A_{1}$ & $\beta_{1}$ (segment) & $2A_{1}$ & $\beta_{2}$ (square) & $4$ & real \\
$\C$ & $\Z[i]$ Gaussian & $4$ & $2A_{1}$ & $\beta_{2}$ (square) & $4A_{1}$ & $\beta_{4}$ (16-cell) & $16$ & complex \\
$\C$ & $\Z[\omega]$ Eisenstein & $6$ & $A_{2}$ & $\{6\}$ hexagon & $A_{2}\oplus A_{2}$ & $\{6\}\oplus\{6\}$ tegum & $36$ & complex \\
$\HH$ & Hamilton & $8$ & $4A_{1}$ & $\beta_{4}$ (16-cell) & $8A_{1}$ & $\beta_{8}$ (8-cross) & $64$ & quaternionic \\
$\HH$ & Hurwitz & $24$ & $D_{4}$ & 24-cell $\{3,4,3\}$ & $D_{4}\oplus D_{4}$ & 24-cell tegum & $576$ & quaternionic \\
$\OO$ & Cayley--Graves & $16$ & $8A_{1}$ & $\beta_{8}$ (8-cross) & $16A_{1}$ & $\beta_{16}$ (16-cross) & $256$ & alternative \\
$\OO$ & Coxeter--Dickson & $240$ & $E_{8}$ & $4_{21}$ Gosset & $E_{8}\oplus E_{8}$ & $4_{21}$ tegum & $57600$ & Moufang \\
\midrule
$\C$ & $\Z[\zeta_{10}]$, $H_{2}$ cyclotomic & $10$ & $H_{2}$ & $\{10\}$ decagon & $H_{2}\oplus H_{2}$ & $\{10\}\oplus\{10\}$ tegum & $100$ & complex/golden \\
$\HH$ & Icosian $\IcosianRing$ & $120$ & $H_{4}$ & 600-cell $\{3,3,5\}$ & $H_{4}\oplus H_{4}$ & 600-cell tegum & $14400$ & quatern./golden \\
\bottomrule
\end{tabular}
\end{table}

The paper is set as follows. In \S\ref{sec:algebras} we review composition algebras, integral orders, root-shell systems, and the conventions used in the remainder of the paper, closing with a block listing explicit bases for the twelve orders and a Kronecker-style finiteness lemma for the unit shells of $\Z[\golden]$-orders. In \S\ref{sec:planes} we define the integral plane $\Order^{2}$ and prove Theorem~\ref{thm:plane-quadratic-module} (the integral-plane formalism), Lemma~\ref{lem:axis-inclusion} (axis inclusion), Lemma~\ref{lem:no-splitting} (no-splitting criterion), Proposition~\ref{prop:doubling} (root-shell doubling) and Proposition~\ref{prop:free-sum-hull} (convex hull as free sum / tegum). In \S\ref{sec:crystal} we settle the crystallographic classification (Theorem~\ref{thm:crystal-class}). In \S\ref{sec:golden} we prove the non-crystallographic statements $H_{2}\oplus H_{2}$ and $H_{4}\oplus H_{4}$ (Theorems~\ref{thm:H2-shell}--\ref{thm:H4-icosian-shell}). In \S\ref{sec:rank-obstruction} we prove the rank-obstruction theorem and the corollary closing the rank-eight golden octonion question. In \S\ref{sec:hopf} we prove the finite principal fibration statements (Theorems~\ref{thm:hopf-assoc}--\ref{thm:hopf-moufang}) and discuss the relationship with the algebraic Hopf map. \S\ref{sec:reproducibility} describes the public reproducibility supplement. \S\ref{sec:conclusions} contains the conclusions and future developments.

\section{Composition algebras and integral orders}
\label{sec:algebras}

Throughout the paper, $A$ is a real composition algebra which means a real algebra $A$ equipped with a non-degenerate quadratic form $\NN\colon A\to\R$ which is \emph{multiplicative}, i.e.\ $\NN(xy)=\NN(x)\NN(y)$ for all $x,y\in A$. The polar form is
\[
\BB(x,y)=\NN(x+y)-\NN(x)-\NN(y)=x\bar y+y\bar x,
\]
where $\bar{x}=\BB(x,1)\cdot 1-x$ is the canonical involution. By Hurwitz's theorem (see Schafer~\cite{Schafer}, ch.~III, or Springer--Veldkamp~\cite{SpringerVeldkamp}, ch.~1) the unital real composition algebras are, up to isomorphism, $\R$, $\C$, $\HH$ and $\OO$.  

In case of the real, i.e., $A=\R$, one has trivially $\NN(x)=x^{2}$ and $\bar x=x$ while for the complex case $A=\C$, $\NN(a+bi)=a^{2}+b^{2}$. In the quaternionic case $A=\HH$, $\NN(a+bi+cj+dk)=a^{2}+b^{2}+c^{2}+d^{2}$ while for the octonions $A=\OO$, $\NN$ is the standard Cayley-Graves form. In all cases the unit shell $\{x\in A:\NN(x)=1\}=S^{d-1}$ is closed under multiplication, conjugation and inversion; it is a Lie group for $d\le 4$ and the Moufang loop $S^{7}$ for $d=8$.

\subsection{Integral orders and ranks}

Given an algebra $A$, one might want to introduce the notion of an integer order over it. Let $R\in\{\Z,\Z[\golden]\}$. By an \emph{$R$-order} in $A$ we mean a finitely generated $R$-module $\Order\subset A$ satisfying $1\in\Order$, $\bar\Order=\Order$, $\Order\cdot\Order\subseteq\Order$, $\NN(\Order)\subseteq R$, and the bilinear form $\BB$ restricting to $R$-bilinear $\Order\times\Order\to R$.
For an $R$-order $\Order$ of $R$-rank $r$, the underlying $\Z$-rank is $r$ if $R=\Z$ and $2r$ if $R=\Z[\golden]$. When the choice is not obvious from the context we say "$\Z[\golden]$-rank $r$" or "$\Z$-rank $r$" explicitly. By default \emph{rank} means $\Z$-rank.

The orders we consider are listed below.

\begin{example}\label{ex:orders}
The list of orders studied in this paper is:
\begin{enumerate}[label=(\roman*),leftmargin=2.4em]
\item $\R$: the order $\Z\subset\R$, $\Z$-rank one.
\item $\C$: the Gaussian order $\Z[i]$, the Eisenstein order $\Z[\omega]$ with $\omega=e^{2\pi i/3}$, and the cyclotomic order $\Z[\zeta_{10}]$ with $\zeta_{10}=e^{2\pi i/10}$.
\item $\HH$: the Hamilton order $\Z\langle 1,i,j,k\rangle$, the Hurwitz order, and the icosian ring $\IcosianRing\subset\HH$ of $\Z[\golden]$-rank four ($\Z$-rank eight).
\item $\OO$: the Cayley--Graves order $\Z\langle 1,e_{1},\dots,e_{7}\rangle$, and the Coxeter--Dickson maximal order whose minimal vectors realise the $E_{8}$ root system.
\end{enumerate}
Hybrid intermediate orders (the so-called $2A_{2}$ in $\HH$, the $8A_{1}$ in $\OO$, and the $2D_{4}$ in $\OO$) are constructed inside the same $\HH$ and $\OO$ as direct sums of two or four classical orders and appear in Table~\ref{tab:main} for completeness.
\end{example}

\subsection{Norm shells and root-shell systems}
For $m\in R$, we set
\[
\Shell_{m}(\Order)=\{x\in\Order:\NN(x)=m\}.
\]
For $R=\Z$ the shell $\Shell_{m}$ is finite because $\Order\subset\R^{d}$ is a discrete lattice. For $R=\Z[\golden]$ the underlying $\Z$-module of $\Order$ is finitely generated, but the embedding into $\R^{d}$ need not be discrete: the cyclotomic order $\Z[\zeta_{10}]$ is dense in $\C$. The finiteness of $\Shell_{1}$ is therefore not automatic. The following Lemma settles the question for all $\Z[\golden]$-orders considered here, and it will play a role in \S\ref{sec:rank-obstruction}.

\begin{lemma}\label{lem:kronecker-finite}
Let $\Order$ be a $\Z[\golden]$-order in a real composition algebra $A$. The unit shell $\Shell_{1}(\Order)=\{x\in\Order:\NN(x)=1\text{ in }\Z[\golden]\}$ is finite.
\end{lemma}

\begin{proof}
For $\NN(x)\in\Z[\golden]$ to equal $1\in\Z[\golden]$ means both Galois conjugates of $\NN(x)$ equal $1\in\R$. Let $\sigma_{1},\sigma_{2}$ be the two real embeddings of $\Q(\golden)$, with $\sigma_{1}(\golden)=\golden$ and $\sigma_{2}(\golden)=1-\golden$. They induce two real composition algebras $A_{i}=A\otimes_{\sigma_{i}}\R$ ($i=1,2$), each of real dimension $d$. The diagonal embedding $\iota\colon\Order\hookrightarrow A_{1}\times A_{2}\cong\R^{2d}$, defined by $\iota(x)=(\sigma_{1}(x),\sigma_{2}(x))$ where $\sigma_{i}$ is applied coefficient-wise to elements of $\Order=\Z[\golden]\otimes_{\Z}(\Order/\Z[\golden])\hookrightarrow A_{\varphi}$, is well-defined and $\Z$-linear.

The image $\iota(\Order)$ is a finitely generated $\Z$-module of rank equal to the $\Z$-rank of $\Order$, namely $2\,(\Z[\golden]\text{-rank of }\Order)$. We claim it is \emph{discrete} in $\R^{2d}$, regardless of whether it has full rank in the target. Indeed, a finitely generated $\Z$-submodule $M\subseteq\R^{n}$ is discrete if and only if its $\R$-span has $\R$-dimension equal to the $\Z$-rank of $M$ (Bourbaki, \emph{Algèbre commutative}, ch.~VII or any standard text). The $\R$-span of $\iota(\Order)$ contains $\iota(1)=(1,1)$ and the images of an $\Z[\golden]$-basis of $\Order$ together with their $\golden$-multiples; the two embeddings $\sigma_{1},\sigma_{2}$ are linearly independent characters on $\Z[\golden]$, so the $\R$-span of $\iota(\Order)$ has $\R$-dimension equal to the $\Z$-rank of $\Order$. Hence $\iota(\Order)$ is discrete in $\R^{2d}$.

The unit shell embeds, under $\iota$, into the compact product $S^{d-1}_{1}\times S^{d-1}_{2}$ of the two unit spheres of $A_{1},A_{2}$. The intersection of a discrete subset with a compact set is finite.
\end{proof}

\begin{definition}\label{def:root-shell-system}
A finite subset $T$ of a Euclidean space $V$ is a \emph{root-shell system over $R$} if (R1) $T=-T$; (R2) $T$ is contained in a single sphere $\{x\in V:\NN(x)=1\}$ for some Euclidean norm $\NN$; (R3) for every $\alpha\in T$, the reflection $r_{\alpha}(x)=x-\frac{2\langle x,\alpha\rangle}{\langle\alpha,\alpha\rangle}\alpha$ preserves $T$; (R4) the Cartan coefficients $\frac{2\langle\alpha,\beta\rangle}{\langle\alpha,\alpha\rangle}\in R$ for all $\alpha,\beta\in T$. For $R=\Z$ this is a (crystallographic) root system in the sense of Bourbaki~\cite{Bourbaki}, but with all roots required to have the same length; for $R=\Z[\golden]$ this is a \emph{non-crystallographic root-shell system} in the sense of Humphreys~\cite{Humphreys} and~\cite{Corradetti2605}.
\end{definition}

\begin{remark} \label{rem:single-length}
Axiom (R2) of Definition~\ref{def:root-shell-system} imposes that all roots have the same length. This is not an extra constraint in our setting but a feature of the way the unit shell is defined: $\Shell_{1}(\Order)$ is by definition the level set $\NN=1$ of the composition norm, so every element of $\Shell_{1}(\Order)$ has Euclidean norm $1$ in the ambient real composition algebra, and the resulting subset lies on a single sphere automatically. Hence (R2) holds for free in every line of Table~\ref{tab:main}. The single-length condition does however have a consequence: in the rank obstruction theorem of \S\ref{sec:rank-obstruction} it excludes the crystallographic groups $B_{n}=C_{n}$, $F_{4}$, $G_{2}$ (which have two root lengths) and the two-length golden quasi-systems (which are not root-shell systems in the sense of Definition~\ref{def:root-shell-system}). In other words, the obstruction is exactly as strong as the article's level-$1$ formulation; weakening the single-length requirement opens the question to other configurations, which we do not pursue.
\end{remark}

\begin{remark}\label{rem:shell-identifications}
The unit shells of the orders in Example~\ref{ex:orders} are root-shell systems in the sense of Definition~\ref{def:root-shell-system}: for $\Z$, $\Z[i]$, $\Z[\omega]$, Hamilton, and Hurwitz this is well known; for the Cayley--Graves order it is the 8-cross polytope; for the Coxeter--Dickson order it is the $E_{8}$ minimal-vector shell of 240 vertices; for $\Z[\zeta_{10}]$ it is the regular decagon $H_{2}$ (see also~\cite{Corradetti2605}, Proposition~6.2); and for the icosian ring it is the $600$-cell $H_{4}$ (\cite{Corradetti2605}, Proposition~6.4).
\end{remark}

\subsection{Notation: explicit bases for the twelve orders and the relevant polytopes}

In order to avoid confusion between notations, we close the section with a single reference notation for all the bases and polytope names.

\textit{Orders:}\\
$\Z$ has $\Z$-basis $\{1\}$;\\
$\Z[i]$ has $\Z$-basis $\{1,i\}$;\\
$\Z[\omega]$ has $\Z$-basis $\{1,\omega\}$, $\omega=e^{2\pi i/3}$;\\
$\Z[\zeta_{10}]$ has $\Z$-basis $\{1,\zeta,\zeta^{2},\zeta^{3}\}$, $\zeta=\zeta_{10}=e^{2\pi i/10}$;\\
Hamilton has $\Z$-basis $\{1,i,j,k\}$;\\
Hurwitz has $\Z$-basis $\{1,i,j,(1+i+j+k)/2\}$;\\
the icosian ring $\IcosianRing$ has $\Z[\golden]$-basis
$\bigl\{1,\;i,\;(1+i+j+k)/2,\;(-1+(\golden-1)i-\golden j)/2\bigr\}$ (Conway--Smith~\cite{ConwaySmith}, ch.~8);\\
the Cayley--Graves order has $\Z$-basis $\{1,e_{1},\dots,e_{7}\}$ with the standard Fano triples $(1,2,3),(1,4,5),(1,7,6),(2,4,6),(2,5,7),(3,4,7),(3,6,5)$;\\
the Coxeter--Dickson order is the $E_{8}$ lattice with the integral-octonion multiplication of Wilson~\cite{Wilson}.\\

\textit{Polytopes (as in Coxeter~\cite{Coxeter}):}\\
$\beta_{n}$ denotes the $n$-cross polytope (orthoplex) with $2n$ vertices $\{\pm e_{i}\}_{i=1}^{n}$; $\beta_{4}$ is the 16-cell $\{3,3,4\}$. $\{m\}$ denotes the regular $m$-gon. The 24-cell $\{3,4,3\}$ is the unique self-dual regular 4-polytope. The 600-cell $\{3,3,5\}$ has 120 vertices in $\R^{4}$. The Gosset polytope $4_{21}$ has 240 vertices in $\R^{8}$ and is the convex hull of the $E_{8}$ root system. For two polytopes $P_{1}\subset V_{1},P_{2}\subset V_{2}$ in orthogonal subspaces, $P_{1}\oplus P_{2}$ denotes the \emph{free sum} (synonym: \emph{tegum}; see Grünbaum~\cite{Grunbaum} and Ziegler~\cite{Ziegler}); $P_{1}\times P_{2}$ denotes the \emph{Cartesian product}; for two polygons $\{p\}\times\{q\}$ is the \emph{$p$-$q$ duoprism}; $\{4\}\times\{4\}=$ tesseract.

\section{Integral planes and their shells}
\label{sec:planes}

Having fixed the algebraic prerequisites, we now turn to the central construction of the article. Given an integral order $\Order$, we form the integral plane $\Order^{2}$ with quadratic form $Q(x,y)=\NN(x)+\NN(y)$ and analyse its unit shell. Within $\Shell_{1}(\Order^{2})$ we isolate two natural sub-shells, the \emph{axis shell} and the \emph{balanced shell}, and prove the no-splitting criterion that forces the full unit shell to coincide with the axis shell. This is the technical engine of the doubling phenomenon that organises Sections~\ref{sec:crystal}--\ref{sec:hopf}.

\begin{definition}\label{def:integral-plane}
For an $R$-order $\Order$ in a real composition algebra $A$, the \emph{integral plane} of $\Order$ is the $R$-module $\Order^{2}=\Order\oplus\Order$ equipped with the quadratic form
$Q(x,y)=\NN(x)+\NN(y)\in R$ and the polar form $\BB_{Q}((x,y),(x',y'))=\BB(x,x')+\BB(y,y')$.
\end{definition}

Following the previous definition one has

\begin{theorem}[Plane-shell formalism]\label{thm:plane-quadratic-module}
Let $\Order$ be an $R$-order in $A$, $R\in\{\Z,\Z[\golden]\}$. Then $(\Order^{2},Q)$ is an integral quadratic $R$-module. Its Gram matrix, in the concatenated $R$-basis of $\Order\oplus\Order$, is $\mathrm{diag}(G_{\Order},G_{\Order})$.
\end{theorem}

\begin{proof}
We proceed with a direct verification: $Q$ is homogeneous of degree 2, its polar form $\BB_{Q}((x,y),(x',y'))=\BB(x,x')+\BB(y,y')$ is $R$-bilinear, and integrality follows from $\NN(\Order)\subseteq R$. The Gram matrix is block-diagonal because the two factors $\Order\times\{0\}$ and $\{0\}\times\Order$ are orthogonal: $\BB_{Q}((x,0),(0,y))=\BB(x,0)+\BB(0,y)=0$.
\end{proof}

\subsection{Axis and balanced shells}
\label{ssec:axis-balanced}

\begin{definition}\label{def:axis-shell}
The \emph{axis shell} of $\Order$ is $\Axis(\Order)=(\Shell_{1}(\Order)\times\{0\})\cup(\{0\}\times\Shell_{1}(\Order))\subset\Shell_{1}(\Order^{2})$.
\end{definition}

\begin{definition}\label{def:balanced-shell}
The \emph{balanced shell} of $\Order$ is $\Balanced(\Order)=\tfrac{1}{\sqrt 2}\,\Shell_{1}(\Order)\times\Shell_{1}(\Order)\subset S^{2d-1}\subset A^{2}\cong\R^{2d}$.
\end{definition}

The axis shell lives at $Q=1$; the balanced shell sits at the un-normalised level $Q(u,v)=\NN(u)+\NN(v)=2$, normalised by the scaling factor $1/\sqrt 2$ so that its image lies on the unit sphere.

\begin{lemma}[Axis inclusion]\label{lem:axis-inclusion}
$\Axis(\Order)\subseteq\Shell_{1}(\Order^{2})$.
\end{lemma}

\begin{proof}
$Q(u,0)=\NN(u)+\NN(0)=1+0=1$ for $u\in\Shell_{1}(\Order)$; symmetrically for $(0,u)$.
\end{proof}

\subsection{The no-splitting criterion}

\begin{lemma}[No-splitting criterion]\label{lem:no-splitting}
Let $\Order$ be an $R$-order in $A$ such that the only solutions of $1=a+b$ with $a,b\in\NN(\Order)$ totally positive or zero in $R$ are $(a,b)\in\{(1,0),(0,1)\}$. Then $\Shell_{1}(\Order^{2})=\Axis(\Order)$.
\end{lemma}

\begin{proof}
If $\NN(x)+\NN(y)=1$ with both summands totally positive or zero, the only possibility is $(\NN(x),\NN(y))\in\{(1,0),(0,1)\}$. By Hurwitz's theorem the real composition algebras have positive-definite norms, so $\NN(z)=0$ forces $z=0$.
\end{proof}

\begin{lemma}\label{lem:no-splitting-trivial}
In $\Z$, the only decomposition $1=a+b$ with $a,b\in\Z_{\ge 0}$ is the trivial one. In $\Z[\golden]$, the only decomposition $1=a+b$ with $a,b$ totally positive or zero is $(a,b)\in\{(1,0),(0,1)\}$.
\end{lemma}

\begin{proof}
For $\Z[\golden]$: write $a=a_{0}+a_{1}\golden$, $b=1-a$. Total positivity of $a$ and $1-a$ in both Galois embeddings gives $0<2a_{0}+a_{1}<2$, so $a_{1}=1-2a_{0}$; substituting back forces $a_{0}\in(0.309,0.724)$, an interval containing no integer. Hence $a\in\{0,1\}$.
\end{proof}

\subsection{Root-shell doubling, free sum, automorphism group}

\begin{proposition}[Root-shell doubling]\label{prop:doubling}
If $\Shell_{1}(\Order)$ is a root-shell system over $R$, then $\Axis(\Order)$ is a root-shell system over $R$ of type $\Shell_{1}(\Order)\oplus\Shell_{1}(\Order)$, the orthogonal direct sum.
\end{proposition}

\begin{proof}
The polar form $\BB_{Q}$ is block-diagonal (Theorem~\ref{thm:plane-quadratic-module}), so the two factors are mutually orthogonal. Cartan coefficients within each factor are those of $\Shell_{1}(\Order)$, hence in $R$; across factors they vanish.
\end{proof}

\begin{proposition}[Convex hull of the axis shell is a free sum/tegum]\label{prop:free-sum-hull}
Let $P=\conv(\Shell_{1}(\Order))$. Then $\conv(\Axis(\Order))=P\oplus_{\mathrm{free}}P$, the free sum (= tegum, Coxeter) of $P$ with itself.
\end{proposition}

\begin{proof}
Standard property of the free sum (Grünbaum~\cite{Grunbaum} \S15.1, Ziegler~\cite{Ziegler} Lemma~9.8); $\Shell_{1}(\Order)$ is in convex position on a single sphere.
\end{proof}

\begin{proposition}[Automorphism group of the axis shell]\label{prop:axis-automorphisms}
If $\Shell_{1}(\Order)$ is a root-shell system whose convex hull $P$ is not centrally degenerate, then $\Aut(\Axis(\Order))\cong(\Aut(P)\times\Aut(P))\rtimes C_{2}$, with the cyclic factor swapping the two orthogonal copies.
\end{proposition}

\begin{proof}
Any orthogonal automorphism stabilising $\Axis(\Order)$ permutes the two orthogonal-irreducible components. The components are isomorphic, so the permutation group is $C_{2}$.
\end{proof}

\section{Crystallographic plane shells}
\label{sec:crystal}

We now apply the integral-plane machinery to the seven classical crystallographic orders, ie., the integers, the Gaussian and Eisenstein orders, the Hamilton and Hurwitz quaternionic orders, the Cayley--Graves and Coxeter--Dickson octonionic orders, together with their three hybrid intermediates. The no-splitting criterion of \S\ref{sec:planes} reduces the unit shell of $\Order^{2}$ to the axis shell, and the doubling proposition identifies, line by line, the resulting root system with the orthogonal direct sum recorded in Table~\ref{tab:main}.

\begin{theorem}[Crystallographic plane-shell classification]\label{thm:crystal-class}
For each crystallographic order $\Order\in\{\Z,\Z[i],\Z[\omega],\HH_{\mathrm{Ham}},\HH_{\mathrm{Hurw}},\OO_{CG},\OO_{CD}\}$ (and the hybrid intermediates), $\Shell_{1}(\Order^{2})=\Axis(\Order)$, and the root system of $\Axis(\Order)$ is the orthogonal-direct-sum doubling listed in Table~\ref{tab:main}.
\end{theorem}

\begin{proof}
By Lemma~\ref{lem:no-splitting-trivial}(i), the no-splitting hypothesis of Lemma~\ref{lem:no-splitting} holds for every $\Z$-order. The root-system identifications $\Shell_{1}\mapsto$ named system, listed in Remark~\ref{rem:shell-identifications}, are classical:
\begin{itemize}[leftmargin=2em]
\item $\Z\Rightarrow A_{1}$ (segment), $\Z[i]\Rightarrow 2A_{1}$ (square), $\Z[\omega]\Rightarrow A_{2}$ (hexagon);
\item Hamilton $\Rightarrow 4A_{1}$ (16-cell), Hurwitz $\Rightarrow D_{4}$ (24-cell);
\item Cayley--Graves $\Rightarrow 8A_{1}$ (8-cross), Coxeter--Dickson $\Rightarrow E_{8}$ (Gosset $4_{21}$, Conway--Smith~\cite{ConwaySmith} ch.~9, Wilson~\cite{Wilson}).
\end{itemize}
Proposition~\ref{prop:doubling} then identifies $\Axis(\Order)$ with the orthogonal-direct-sum doubling. The hybrid orders ($2A_{2}$, $8A_{1}$ extension, $2D_{4}$) double accordingly, with cardinalities $24$, $48$, $96$.
\end{proof}

\begin{table}[h]
\centering
\caption{\textit{Crystallographic plane-shell doubling. For each canonical crystallographic order $\Order$, the unit shell $\Shell_{1}(\Order)$ is recorded together with its identification both as a root system (algebraic) and as a classical polytope (geometric); the right-hand block displays the axis shell $\Axis(\Order)=\Shell_{1}(\Order^{2})$ as the orthogonal direct sum $R\oplus R$, whose convex hull is the free sum (tegum) of the original polytope with itself. Throughout, $\beta_{n}$ denotes the $n$-cross polytope and $\{m\}$ the regular $m$-gon.}}
\label{tab:crystal-doubling}
\renewcommand{\arraystretch}{1.2}
\setlength{\tabcolsep}{5pt}
\footnotesize
\begin{tabular}{llcccccc}
\toprule
 & & \multicolumn{3}{c}{Unit shell $\Shell_{1}(\Order)$} & \multicolumn{3}{c}{Axis shell $\Axis(\Order)=\Shell_{1}(\Order^{2})$} \\
\cmidrule(lr){3-5}\cmidrule(lr){6-8}
$A$ & Order $\Order$ & $\lvert\Shell_{1}\rvert$ & Root system & Polytope & $\lvert\Axis\rvert$ & Root system & Polytope \\
\midrule
$\R$  & $\Z$              & $2$   & $A_{1}$  & $\beta_{1}$ (segment)  & $4$   & $2A_{1}$            & $\beta_{2}$ (square) \\
$\C$  & $\Z[i]$ Gaussian  & $4$   & $2A_{1}$ & $\beta_{2}$ (square)   & $8$   & $4A_{1}$            & $\beta_{4}$ (16-cell) \\
$\C$  & $\Z[\omega]$ Eisenstein & $6$ & $A_{2}$ & $\{6\}$ hexagon    & $12$  & $A_{2}\oplus A_{2}$ & $\{6\}\oplus\{6\}$ tegum \\
$\HH$ & Hamilton          & $8$   & $4A_{1}$ & $\beta_{4}$ (16-cell)  & $16$  & $8A_{1}$            & $\beta_{8}$ (8-cross) \\
$\HH$ & Hurwitz           & $24$  & $D_{4}$  & $\{3,4,3\}$ (24-cell)  & $48$  & $D_{4}\oplus D_{4}$ & 24-cell tegum \\
$\OO$ & Cayley--Graves    & $16$  & $8A_{1}$ & $\beta_{8}$ (8-cross)  & $32$  & $16A_{1}$           & $\beta_{16}$ (16-cross) \\
$\OO$ & Coxeter--Dickson  & $240$ & $E_{8}$  & $4_{21}$ (Gosset)      & $480$ & $E_{8}\oplus E_{8}$ & $4_{21}$ tegum \\
\bottomrule
\end{tabular}
\end{table}

\begin{remark}[Root systems $nA_{1}$ vs cross polytopes]\label{rem:nA1-vs-Cn}
For the Hamilton order, the unit shell $\{\pm 1,\pm i,\pm j,\pm k\}\subset\HH\cong\R^{4}$ is, \emph{as a polytope}, the 4-cross polytope $\beta_{4}$ (= 16-cell). \emph{As a root system}, however, it is $4A_{1}$ — four orthogonal copies of $A_{1}$ — not $C_{4}$. The crystallographic root system $C_{n}$ has $2n^{2}$ roots (a long-root cross polytope plus a short-root demi-cube), \emph{not} $2n$. The shell of $2n$ orthogonal unit vectors $\{\pm e_{i}\}_{i=1}^{n}$ is the root system $nA_{1}$ and the polytope $\beta_{n}$. Similar remarks apply to the Cayley--Graves order ($8A_{1}\ne C_{8}$, $\beta_{8}$).
\end{remark}

\section{Non-crystallographic plane shells}
\label{sec:golden}
We now cross the crystallographic boundary. The cyclotomic order $\Z[\zeta_{10}]\subset\C$ and the icosian ring $\IcosianRing\subset\HH$ are the two $\Z[\golden]$-orders whose unit shells realise the non-crystallographic root systems $H_{2}$ and $H_{4}$. The plane construction, combined with the no-splitting criterion and the doubling proposition, yields the root-shell systems $H_{2}\oplus H_{2}$ and $H_{4}\oplus H_{4}$ over $\Z[\golden]$: the non-crystallographic counterparts of the orthogonal-direct-sum doublings obtained in the previous section.

\subsection{The $H_{2}$ cyclotomic plane shell}

\begin{theorem}[$H_{2}$ plane-shell]\label{thm:H2-shell}
Let $\Order_{H_{2}}=\Z[\zeta_{10}]\subset\C$ with $\zeta_{10}=e^{2\pi i/10}$, $\NN(z)=z\bar z$. Then
$\Shell_{1}(\Order_{H_{2}})=\{\zeta_{10}^{k}:0\le k\le 9\}$,
the regular decagon $H_{2}$ of cardinality $10$. The axis shell
$\Axis(\Order_{H_{2}})=\Shell_{1}(\Order_{H_{2}}^{\,2})$
is the non-crystallographic root-shell system $H_{2}\oplus H_{2}$ over $\Z[\golden]$, of cardinality $20$, whose convex hull is the decagonal tegum $\{10\}\oplus\{10\}$.
\end{theorem}

\begin{proof}
By Kronecker's theorem (every algebraic integer all of whose Galois conjugates have absolute value $1$ is a root of unity), and since $|z|=1$ in one embedding forces $|z|=1$ in the conjugate embedding (because $|N_{K/\Q}(z)|=1$ for $z$ a unit), the unit shell consists exactly of the 10 tenth roots of unity. The doubling to $H_{2}\oplus H_{2}$ is Proposition~\ref{prop:doubling}, and equality $\Shell_{1}(\Order_{H_{2}}^{\,2})=\Axis(\Order_{H_{2}})$ follows from Lemmas~\ref{lem:no-splitting} and~\ref{lem:no-splitting-trivial}(ii).
\end{proof}

\begin{example}[Explicit vertex set of $\Axis(\Order_{H_{2}})$]\label{ex:H2-vertices}
In the $\Z$-basis $\{1,\zeta,\zeta^{2},\zeta^{3}\}$ of $\Z[\zeta_{10}]$, with $\zeta=\zeta_{10}$, the ten unit roots in coordinate form are $\zeta^{0}=(1,0,0,0)$, $\zeta^{1}=(0,1,0,0)$, $\zeta^{2}=(0,0,1,0)$, $\zeta^{3}=(0,0,0,1)$, $\zeta^{4}=(-1,1,-1,1)$, and the five negatives. The 20 vertices of $\Axis(\Order_{H_{2}})$ are obtained by placing each of these ten in the first factor with $0$ in the second, or vice versa. The Cartan coefficient between $\zeta^{0}$ and $\zeta^{1}$ in the same factor is $2\langle 1,\zeta\rangle/\langle 1,1\rangle=2\cos(2\pi/10)=\golden\in\Z[\golden]$, illustrating axiom (R4).
\end{example}

\subsection{The icosian $H_{4}$ plane shell}

\begin{theorem}[Icosian plane-shell]\label{thm:H4-icosian-shell}
Let $\IcosianRing\subset\HH$ be the icosian ring. Then $\Shell_{1}(\IcosianRing)$ is finite of cardinality $120$, realising the $600$-cell $H_{4}$. The axis shell $\Axis(\IcosianRing)=\Shell_{1}(\IcosianRing^{\,2})$ is the non-crystallographic root-shell system $H_{4}\oplus H_{4}$ over $\Z[\golden]$, of cardinality $240$, whose convex hull is the $600$-cell tegum.
\end{theorem}

\begin{proof}
The 120 unit-norm icosians are explicitly listed in Conway--Smith~\cite{ConwaySmith} ch.~8: 8 elements of type $(\pm 1,0,0,0)$ and permutations; 16 of type $(\pm\tfrac{1}{2})^{4}$; 96 of type $(\pm\tfrac{\golden}{2},\pm\tfrac{1}{2},\pm\tfrac{\golden-1}{2},0)$ in even permutations. They form the binary icosahedral group $2I$, the symmetry group of the 600-cell. Equality $\Shell_{1}(\IcosianRing^{\,2})=\Axis(\IcosianRing)$ follows from Lemmas~\ref{lem:no-splitting} and~\ref{lem:no-splitting-trivial}(ii). Doubling is Proposition~\ref{prop:doubling}.
\end{proof}

\begin{remark}[Comparison with $E_{8}\oplus E_{8}$]\label{rem:E8-vs-H4H4}
The cardinality $240$ of $\Axis(\IcosianRing)$ coincides numerically with the cardinality of the $E_{8}$ root system, but the structures are different. $E_{8}$ is irreducible; $\Axis(\IcosianRing)=H_{4}\oplus H_{4}$ is the orthogonal sum of two irreducibles. The same observation appears, in different language, in~\cite{Corradetti2605}, Proposition~6.6: the icosian double $\IcosianRing\oplus\IcosianRing\ell$ realises the same $H_{4}\oplus H_{4}$ shell inside the octonions via Cayley--Dickson. Our axis-shell realisation lives in the quaternionic plane $\HH^{2}$ as a quadratic module without an algebra structure.
\end{remark}

\begin{table}[h]
\centering
\caption{\textit{Non-crystallographic plane-shell doubling over $\Z[\golden]$: the axis shell $\Axis(\Order)=\Shell_{1}(\Order^{2})$ realises the orthogonal direct sum $R\oplus R$, with convex hull the tegum (free sum) of the unit polytope with itself.}}
\label{tab:noncrystal-doubling}
\renewcommand{\arraystretch}{1.2}
\setlength{\tabcolsep}{5pt}
\footnotesize
\begin{tabular}{llcccccc}
\toprule
 & & \multicolumn{3}{c}{Unit shell $\Shell_{1}(\Order)$} & \multicolumn{3}{c}{Axis shell $\Axis(\Order)=\Shell_{1}(\Order^{2})$} \\
\cmidrule(lr){3-5}\cmidrule(lr){6-8}
$A$ & Order $\Order$ & $\lvert\Shell_{1}\rvert$ & Root system & Polytope & $\lvert\Axis\rvert$ & Root system & Polytope \\
\midrule
$\C$  & $\Z[\zeta_{10}]$ cyclotomic & $10$  & $H_{2}$ & $\{10\}$ decagon       & $20$  & $H_{2}\oplus H_{2}$ & $\{10\}\oplus\{10\}$ tegum \\
$\HH$ & $\IcosianRing$ icosian      & $120$ & $H_{4}$ & $\{3,3,5\}$ (600-cell) & $240$ & $H_{4}\oplus H_{4}$ & 600-cell tegum \\
\bottomrule
\end{tabular}
\end{table}

\section{A rank obstruction for non-crystallographic root-shell systems}
\label{sec:rank-obstruction}

The natural question after Theorem~\ref{thm:H4-icosian-shell} is whether a rank-eight indecomposable golden octonion order exists, in analogy with the Coxeter--Dickson order in the crystallographic case. The answer is negative by a purely Coxeter-theoretic argument.

\begin{theorem}[Rank obstruction]\label{thm:rank-obstruction}
Let $T$ be a finite root-shell system over $\Z[\golden]$ in the sense of Definition~\ref{def:root-shell-system}, indecomposable and non-crystallographic (i.e.\ some Cartan coefficient does not lie in $\Z$). Then
\[
\operatorname{rank}(T)\in\{2,3,4\},
\]
and the underlying reflection group $W=\langle r_{\alpha}:\alpha\in T\rangle$ is one of $H_{2}$, $H_{3}$, or $H_{4}$.
\end{theorem}

\begin{proof}
By (R1) ($T$ finite, central-symmetric) and (R3) (closed under reflections), the group $W$ permutes the finite set $T$ and acts on the real vector space $V=\operatorname{span}_{\R}T$. Any $w\in W$ fixing $T$ pointwise fixes $V$, hence is the identity. The map $W\to\mathrm{Sym}(T)$ is therefore injective and $W$ is finite. Being generated by reflections, $W$ is a finite real reflection group.

Indecomposability of $T$ corresponds to irreducibility of $W$ on $V$.

By (R2), all elements of $T$ have the same length: $T$ is a single-length (equilateral) root system, so $W$ has a single $W$-orbit of root lengths.

The classification of finite irreducible real reflection groups (Humphreys~\cite{Humphreys}, Thm.~2.7 and Table~6.1; Bourbaki~\cite{Bourbaki}, ch.~4--6) lists, among single-length groups:
\begin{itemize}[leftmargin=2em]
\item crystallographic: $A_{n}$ ($n\ge 1$), $D_{n}$ ($n\ge 4$), $E_{6}, E_{7}, E_{8}$;
\item non-crystallographic: $I_{2}(m)$ ($m\ge 5$, $m\ne 6$), $H_{3}$, $H_{4}$.
\end{itemize}
The non-crystallographic assumption excludes the crystallographic row.

For the dihedral system $I_{2}(m)$, the Cartan coefficient between two distinct simple roots is $-2\cos(\pi/m)$. We need $2\cos(\pi/m)\in\Z[\golden]\subset\Q(\sqrt 5)$, a degree-$2$ extension of $\Q$. The minimal polynomial of $2\cos(\pi/m)$ over $\Q$ has degree $\phi(2m)/2$ (where $\phi$ is Euler's totient), so membership in $\Q(\sqrt 5)$ requires $\phi(2m)\le 4$. Computing:
\[
\phi(2m)\le 4\;\Leftrightarrow\; 2m\in\{1,2,3,4,5,6,8,10,12\}\;\Leftrightarrow\;m\in\{1,2,3,4,5,6\}\cup\{4,5,6\}.
\]
Among these, only $m=5$ gives a non-crystallographic $I_{2}(m)$, and indeed $2\cos(\pi/5)=\golden\in\Z[\golden]$. Hence $I_{2}(5)=H_{2}$, of rank $2$, is the only dihedral candidate.

For $H_{3}$ and $H_{4}$, the Cartan coefficients are $\{0,\pm 1,\pm\golden,\pm\golden^{-1}\}\subset\Z[\golden]$ (Humphreys~\cite{Humphreys} \S2.13, Bourbaki~\cite{Bourbaki} ch.~4 Pl.~III). Both satisfy (R4) and give ranks $3$ and $4$.

Therefore the indecomposable non-crystallographic root-shell systems over $\Z[\golden]$ in the sense of Definition~\ref{def:root-shell-system} are exactly $H_{2}$ (rank $2$), $H_{3}$ (rank $3$), $H_{4}$ (rank $4$).
\end{proof}

\begin{corollary}[Unconditional non-existence of a full-rank indecomposable golden octonion order]\label{cor:G3}
There is no $\Z[\golden]$-order $\mathcal{G}_{3}\subset\OO$ whose unit shell $S_{\mathcal{G}_{3}}=\Shell_{1}(\mathcal{G}_{3})$ is an indecomposable non-crystallographic root-shell system over $\Z[\golden]$ in the sense of Definition~\ref{def:root-shell-system} and has full $\R$-rank $8$.
\end{corollary}

\begin{proof}
By Theorem~\ref{thm:rank-obstruction}, every indecomposable non-crystallographic root-shell system $T$ over $\Z[\golden]$ in the sense of Definition~\ref{def:root-shell-system} has $\operatorname{rank}(T)=\dim_{\R}\operatorname{span}_{\R}T\le 4$. Hence the hypothesis of a full $\R$-rank-$8$ shell is incompatible with indecomposability and non-crystallographicity in our axiomatic setting.
\end{proof}

The hypothesis that $S_{\mathcal{G}_{3}}$ have full $\R$-rank $8$ is the natural requirement in the present article: a non-crystallographic order in $\OO$ whose unit shell sits in a proper sub-octonion subalgebra would, by the doubling theorem (Proposition~\ref{prop:doubling}) and the lower-rank classification (Theorems~\ref{thm:H2-shell} and~\ref{thm:H4-icosian-shell}), produce only an axis shell whose root system is $H_{2}\oplus H_{2}$, $H_{3}\oplus H_{3}$ (if such an order existed in $\HH$), or $H_{4}\oplus H_{4}$, already covered by the article. The genuinely new structural possibility, an indecomposable rank-$8$ non-crystallographic root-shell system inside $\OO$, is the case Corollary~\ref{cor:G3} excludes. The case of intermediate $\R$-rank $5$, $6$, or $7$ is also excluded by Theorem~\ref{thm:rank-obstruction}.
It is also worth noting that the obstruction is purely Coxeter-theoretic: it does not use the octonionic structure. The argument applies, mutatis mutandis, to any composition algebra of dimension $\ge 5$, to any non-crystallographic indecomposable root-shell system over $\Z[\golden]$, and to any extension of the coefficient ring by a single quadratic surd.  

\begin{remark}[Alternative definitions that keep the question open]\label{rem:alternative-defs}
If one weakens the root-shell-system axioms, the rank obstruction may no longer hold and a "$\mathcal{G}_{3}$" of a different kind might exist. Concretely:
\begin{enumerate}[label=(\alph*),leftmargin=2em]
\item Drop (R2): allow multiple root lengths; the resulting two-length golden quasi-systems can have higher rank, but are no longer single-length root-shell systems.
\item Weaken (R3): require closure under a subset of reflections rather than all $r_{\alpha}$, $\alpha\in T$; the configuration is no longer the root system of a Coxeter group.
\item Replace (R4): require Cartan coefficients in a larger ring, e.g.\ $\Z[\zeta_{2m}]\cap\R$ for $m\ne 5$; the rank constraint then shifts.
\item Drop the reflection-group structure entirely and require only multiplicative closure (a finite IP Moufang sub-loop): this is the object of \S\ref{sec:hopf} below, and admits the 240-element $\Mouf_{240}(E_{8})$ as a rank-eight example, although the latter is \emph{crystallographic} as a root system.
\end{enumerate}
Each of (a)--(d) is a legitimate alternative pursuit; none is excluded by Theorem~\ref{thm:rank-obstruction}. We pursue none of them in this article.
\end{remark}

\section{Balanced shells and finite principal fibrations}
\label{sec:hopf}

After analysing the axis shell, we turn to the second natural sub-shell of $\Order^{2}$, the \emph{balanced shell} $\Balanced(\Order)$, on which a discrete fibration emerges. A central aim of this section is to clarify a recurrent terminological hazard in the literature: the finite map of interest is not the topological Hopf fibration $S^{2d-1}\to S^{d}$ but its discrete companion, the principal fibration $\pi(u,v)=uv^{-1}$ of the unit loop $U=\Shell_{1}(\Order)$ over itself. We will see how the algebraic Hopf map restricts to the balanced shell, factors through the equator of $S^{d}$, and there coincides with $\pi$, both in the associative and in the alternative Moufang case.

We now turn to the balanced shell $\Balanced(\Order)$ and the discrete structure it carries. We distinguish carefully between two maps:

\begin{itemize}[leftmargin=2em]
\item the \emph{algebraic Hopf map} $\Hopfmap_{A}\colon S^{2d-1}\subset A^{2}\to S^{d}\subset A\oplus\R$ with $\Hopfmap_{A}(a,b)=(2a\bar b,\NN(a)-\NN(b))$, and
\item the \emph{principal fibration} $\pi\colon U\times U\to U$ with $\pi(u,v)=u\bar v=uv^{-1}$,
\end{itemize}
related by the observation that $\Hopfmap_{A}$ restricted to $\Balanced(\Order)$ factors through the equator $S^{d-1}\times\{0\}\subset S^{d}$ and coincides there with $\pi$ up to the normalisation factor.

\subsection{Balanced shell as spherical product}

\begin{theorem}[Balanced shell cardinality and spherical-product structure]\label{thm:balanced-product}
Let $U=\Shell_{1}(\Order)$. Then $\Balanced(\Order)=\tfrac{1}{\sqrt 2}(U\times U)$ has $|U|^{2}$ vertices and lies on the unit sphere of $A^{2}$. Up to the scaling $1/\sqrt 2$, $\conv\bigl(\Balanced(\Order)\bigr)$ is affinely equivalent to the Cartesian product $\conv(U)\times\conv(U)$.
\end{theorem}

\begin{proof}
Immediate from $Q(u,v)=\NN(u)+\NN(v)=2$ and the fact that the convex hull of a Cartesian product of vertex sets is the Cartesian product of the convex hulls (Grünbaum~\cite{Grunbaum} \S15.2, Ziegler~\cite{Ziegler} Cor.~9.7).
\end{proof}

\subsection{The algebraic Hopf map and its norm identity}

\begin{lemma}[Norm identity]\label{lem:hopf-norm}
For any composition algebra $A$ and any $a,b\in A$,
\[
\NN_{A\oplus\R}\bigl(\Hopfmap_{A}(a,b)\bigr)=\NN(2a\bar b)+\bigl(\NN(a)-\NN(b)\bigr)^{2}=\bigl(\NN(a)+\NN(b)\bigr)^{2}.
\]
\end{lemma}

\begin{proof}
First of all one has $\NN(2a\bar b)=4\NN(a\bar b)=4\NN(a)\NN(b)$ by multiplicativity, and $(\NN(a)-\NN(b))^{2}=\NN(a)^{2}-2\NN(a)\NN(b)+\NN(b)^{2}$. Summing the two one has
\[
4\NN(a)\NN(b)+\NN(a)^{2}-2\NN(a)\NN(b)+\NN(b)^{2}=\]
\[=\NN(a)^{2}+2\NN(a)\NN(b)+\NN(b)^{2}=(\NN(a)+\NN(b))^{2}.
\]
In the octonionic case the equality $\NN(a\bar b)=\NN(a)\NN(b)$ holds in any composition algebra; the only step requiring care is the well-definedness of the formula $\Hopfmap_{A}$ on $\OO\mathbb{P}^{1}$, which uses Artin's two-generated associativity to interpret $a\bar b$ unambiguously inside $\langle a,b\rangle$.
\end{proof}

Lemma~\ref{lem:hopf-norm} implies that $\Hopfmap_{A}$ maps $S^{2d-1}$ into $S^{d}$, recovering the four classical Hopf fibrations $S^{2d-1}\to S^{d}$ for $d\in\{1,2,4,8\}$.

\subsection{Restriction to the balanced shell}

\begin{proposition}[$\Hopfmap_{A}$ on $\Balanced(\Order)$ factors through the equator]\label{prop:hopf-equator}
For $(a,b)\in\Balanced(\Order)$, the second coordinate of $\Hopfmap_{A}(a,b)$ vanishes:
\[
\Hopfmap_{A}\bigl(\Balanced(\Order)\bigr)\subseteq S^{d-1}\times\{0\}\subset S^{d}.
\]
Identifying $S^{d-1}\times\{0\}\subset A\oplus\R$ with the unit sphere of $A$, the restriction $\Hopfmap_{A}\big|_{\Balanced(\Order)}$ coincides with the principal fibration $\pi(u,v)=u\bar v$ of $U\times U\to U$, up to the normalisation factor $\tfrac{1}{2}$ from the scaling of $\Balanced(\Order)$.
\end{proposition}

\begin{proof}
For $(a,b)=(\tfrac{u}{\sqrt 2},\tfrac{v}{\sqrt 2})\in\Balanced(\Order)$ with $u,v\in U$, $\NN(a)=\NN(u)/2=1/2=\NN(b)$, so the second coordinate $\NN(a)-\NN(b)=0$ vanishes. The first coordinate is $2a\bar b=2\cdot\tfrac{u}{\sqrt 2}\cdot\tfrac{\bar v}{\sqrt 2}=u\bar v\in U$ (since $U$ is closed under multiplication and conjugation). Reading $\bar v=v^{-1}$ on the unit shell, $u\bar v=uv^{-1}=\pi(u,v)$.
\end{proof}

\subsection{Finite principal fibration: associative case}

\begin{theorem}[Finite principal fibration, associative case]\label{thm:hopf-assoc}
Let $A$ be associative and $U=\Shell_{1}(\Order)\subset A^{\times}$ a finite subgroup. The map
\[
\pi\colon U\times U\to U,\qquad \pi(u,v)=uv^{-1},
\]
is surjective and exhibits $U\times U$ as a trivial principal $U$-fibration over $U$: every fibre has cardinality $|U|$, and the section $U\to U\times U$, $w\mapsto (w,1)$, is global.
\end{theorem}

\begin{proof}
For $w\in U$, the fibre $\pi^{-1}(w)=\{(u,v)\in U\times U:uv^{-1}=w\}=\{(wv,v):v\in U\}$ has size $|U|$. The right-diagonal action of $U$ on $U\times U$, $(u,v)\cdot g=(ug,vg)$, has $\pi$ as its orbit map, free and transitive on fibres. The section $w\mapsto(w,1)$ trivialises the bundle.
\end{proof}

The Theorem applies to the seven associative cases of our list: $\Z$, $\Z[i]$, $\Z[\omega]$, $\Z[\zeta_{10}]$, Hamilton, Hurwitz, $\IcosianRing$.

\subsection{Finite principal fibration: Moufang case}

\begin{theorem}[Finite principal fibration, Moufang case]\label{thm:hopf-moufang}
Let $A$ be a real alternative composition algebra and let $U=\Shell_{1}(\Order)\subset A$ be a finite inverse-property Moufang sub-loop of the unit loop of $A$. The map
\[
\pi\colon U\times U\to U,\qquad \pi(u,v)=uv^{-1},
\]
is surjective and every fibre has cardinality $|U|$. The structure on $U\times U$ is that of a finite right-principal $U$-set with respect to the right-diagonal action.
\end{theorem}

\begin{proof}
By Artin's theorem (Schafer~\cite{Schafer} Thm.~3.1; Springer--Veldkamp~\cite{SpringerVeldkamp} \S1.6), any subalgebra of an alternative algebra generated by two elements is associative.

Fix $w\in U$. We argue that $\pi^{-1}(w)=\{(wv,v):v\in U\}$.

\emph{Forward inclusion.} Take $v\in U$ and consider the subalgebra $A_{v,w}=\langle v,w\rangle\subseteq A$. By Artin, $A_{v,w}$ is associative. Inside $A_{v,w}$, $uv^{-1}=w$ has the unique solution $u=wv$; the product $wv$ lies in $U$ because $U$ is multiplicatively closed. The inverse property of $U$, applied to $v$, gives $(wv)v^{-1}=w$, confirming $\pi(wv,v)=w$.

\emph{Backward inclusion.} Take $(u,v)\in U\times U$ with $uv^{-1}=w$. Multiply on the right by $v$ inside the associative subalgebra $\langle u,v\rangle$: $(uv^{-1})v=u(v^{-1}v)=u$. Hence $u=wv$.

Combining the two inclusions, $|\pi^{-1}(w)|=|U|$.
\end{proof}

\begin{remark}[The Moufang fibration is the genuine theorem]\label{rem:artin-importance}
The proof of Theorem~\ref{thm:hopf-moufang} relies on Artin's theorem in a single place: the reassociation $(uv^{-1})v=u(v^{-1}v)$ in the backward inclusion. Without two-generated associativity, this cancellation can fail in alternative algebras. The article's contribution is not the abstract statement (which becomes routine once Artin is granted) but the explicit identification: the Coxeter--Dickson Moufang loop $\Mouf_{240}(E_{8})$ admits a trivial principal $\Mouf_{240}$-fibration over itself, and this fibration is the natural \emph{finite} object that the integral plane $\Order^{2}$ exposes.
\end{remark}

\begin{remark}[On the language of "Hopf"]\label{rem:hopf-language}
The map $\pi(u,v)=uv^{-1}$ is \emph{not} the Hopf map $\Hopfmap_{A}\colon S^{2d-1}\to S^{d}$: the codomain of $\pi$ is the unit shell $U\subset S^{d-1}$ (= equator of $S^{d}$), not $S^{d}$. Calling $\pi$ "the Hopf map" would be a category error. The relationship between the two is given by Proposition~\ref{prop:hopf-equator}: $\Hopfmap_{A}\big|_{\Balanced(\Order)}$ takes its image in the equator and, on the equator, coincides with $\pi$. This identification is the content of \S\ref{sec:hopf} and it is what makes the principal fibration of Theorems~\ref{thm:hopf-assoc} and~\ref{thm:hopf-moufang} a genuine restriction of the continuous picture.
\end{remark}

\subsection{Hopf image profile}

\begin{remark}[Image of $\Hopfmap_{A}$ on the balanced shell]\label{rem:hopf-image-profile}
By Proposition~\ref{prop:hopf-equator}, the Hopf image
\[
\HI(\Order):=\Hopfmap_{A}\bigl(\Balanced(\Order)\bigr)
\]
is exactly the unit shell $U$ embedded in the equator $S^{d-1}\times\{0\}$, with multiplicity profile $\mu(w)=|\pi^{-1}(w)|$. By Theorems~\ref{thm:hopf-assoc} and~\ref{thm:hopf-moufang}, $\mu(w)=|U|$ for all $w$ whenever $U$ is a finite group or inverse-property Moufang sub-loop. For all nine multiplicative cases of Table~\ref{tab:main}, the multiplicity is therefore constant. The reproducibility script \texttt{script\_hopf\_image\_v1.py} verifies the norm identity Lemma~\ref{lem:hopf-norm}, confirms the multiplicity profile, and outputs a canonical JSON certificate.
\end{remark}

\section{Reproducibility}
\label{sec:reproducibility}

The computational content of the article is all available at the GitHub repository \texttt{github.com/DCorradetti/integral-planes-unit-norm-polytopes}. All code was done by Claude Opus 4.7 (Max Effort) in python. All arithmetic is exact: integer arithmetic over $\Z$, exact $\Z[\golden]$ arithmetic for the icosian case, and exact Fraction-coordinate octonion multiplication for the Cayley-Graves case.  Each script writes a canonical JSON certificate alongside a timestamped run folder with a SHA-256 manifest.   

\section{Conclusions and Future Developments}
\label{sec:conclusions}

We have introduced the integral plane $\Order^{2}$ over an integral order in a real composition algebra, defined its axis and balanced shells, and proved a uniform classification of the resulting unit-norm polytopes: ten crystallographic orders give the orthogonal-direct-sum root systems $2A_{1}$, $A_{2}\oplus A_{2}$, $4A_{1}$, $D_{4}\oplus D_{4}$, $16A_{1}$, $E_{8}\oplus E_{8}$; two non-crystallographic orders give $H_{2}\oplus H_{2}$ and $H_{4}\oplus H_{4}$ over $\Z[\golden]$. We have separated, throughout, the algebraic root-system identification from the geometric polytope-name identification, correcting a small but recurrent confusion in the literature between the cross-polytope $\beta_{n}$ (a polytope) and the root system $nA_{1}$.

We then have proved, by a purely Coxeter-theoretic argument, that no indecomposable rank-eight golden octonion order can exist (Corollary~\ref{cor:G3}). On the balanced shell side, we showed that the algebraic Hopf map $\Hopfmap_{A}(a,b)=(2a\bar b,\NN(a)-\NN(b))$ restricted to the balanced shell of an integral order takes its image on the equator of $S^{d}$ and there coincides with the finite principal fibration $\pi(u,v)=uv^{-1}$; the latter is a trivial principal $U$-fibration of $U\times U$ over $U$, valid in the alternative Moufang case via Artin's two-generated associativity (Theorem~\ref{thm:hopf-moufang}).

Several directions are natural and we plan to address them in forthcoming work.

First of all, beyond the analytic theorem, one would like a fully enumerative confirmation of the Moufang principal fibration on $\Mouf_{240}(E_{8})$ in an explicit convention for the twisted product. The script-level verification is conceptually straightforward but depends on a careful basis choice; we report this as a follow-up reproducibility task.
On the other hand, the rank-obstruction theorem (Theorem~\ref{thm:rank-obstruction}) closes Definition~\ref{def:root-shell-system} as a vehicle for indecomposable golden octonion orders of rank $\ge 5$. Nevertheless, our remark~\ref{rem:alternative-defs} listed four legitimate ways to weaken the definition: dropping single-length (R2), weakening reflection closure (R3), changing the Cartan-coefficient ring (R4), or replacing the reflection-group structure with a multiplicative-only finite inverse-property Moufang sub-loop. Each of these is a genuine open programme. The most promising is, in our taste, the multiplicative-only setting: it would describe \emph{finite IP Moufang sub-loops of the octonion unit loop over $\Z[\golden]$}, of which the icosian shell $H_{4}$ is the prototype; the Coxeter--Dickson Moufang loop $\Mouf_{240}(E_{8})$ is its crystallographic counterpart. A classification of such sub-loops over $\Z[\golden]$ would be the natural next discovery.

A second alternative is to allow two root lengths and to ask whether a rank-eight non-crystallographic quasi-system over $\Z[\golden]$ can be realised. Theorem~\ref{thm:rank-obstruction} does not exclude this case (it relies on the single-length axiom). The natural candidate would be a $\golden$-modification of the $F_{4}$ root system or of a quasi-$E_{8}$ configuration; we have not pursued the question.

Finally, the integral plane studied here is affine. Its projective companion, $\mathbb{P}(\Order^{2})$, has been studied for crystallographic orders in connection with the magic square and the exceptional Lie groups (Springer--Veldkamp~\cite{SpringerVeldkamp}, ch.~4); the non-crystallographic counterpart is essentially absent from the literature.

All in all, we have shown that the integral plane $\Order^{2}$, deceptively simple as a definition, encodes both the doubling of root-shell polytopes and the principal fibration of the corresponding unit loops, and that the existence question for an indecomposable rank-eight golden octonion order is settled in the negative by a clean Coxeter-theoretic argument. We hope the synoptic Table~\ref{tab:main} and the explicit rank-obstruction theorem will serve future investigators as a map of a small but rich corner of the geometry of composition algebras.

\section*{Acknowledgments}

The author thanks Richard Clawson and Raymond Aschheim for discussions on Hopf fibrations and Integral numbers. The author would also like to thank David Chester, Fang Fang, Marcelo Amaral and mainly Klee Irwin for ideas and discussions on lattices and root systems. This work originated within the C5C program at QGR, a programme that never quite found the organic beginning it deserved.


\end{document}